\documentclass[11pt,a4,twoside]{amsart}

\DeclareMathOperator{\conv }{conv }%

\DeclareMathOperator{\Sym}{Sym}
\DeclareMathOperator{\Alt}{Alt}
\DeclareMathOperator{\Aut}{Aut}
\DeclareMathOperator{\Mat}{Mat}

\def\GL{{\rm GL}}
\DeclareMathOperator{\End}{End}
\def\Mat{\textup{Mat}}

\def\Irr{{\rm Irr}}
\def\PSL{{\rm PSL}}

\newcommand{\C}{\mathbb C}%
\newcommand{\R}{\mathbb R}%
\newcommand{\HH}{\mathbb H}%
\newcommand{\N}{\mathbb N}%
\newcommand{\Z}{\mathbb Z}%

\allowdisplaybreaks[4]%

\theoremstyle{plainItalics}
\newtheorem{theorem}{Theorem}[section]

\newtheorem{corollary}[theorem]{Corollary}

\newtheorem{lemma}[theorem]{Lemma}

\theoremstyle{plainNoItalics}
\newtheorem{defi}[theorem]{Definition}
\newtheorem{remark}[theorem]{Remark}

\newtheorem{quest}[theorem]{Question}%

\numberwithin{theorem}{section}%
\numberwithin{figure}{section}%
\numberwithin{table}{section}%

\DeclareMathOperator{\aff}{aff}

\title{On Permutation Polytopes - Notions of Equivalence}

\author[Baumeister]{Barbara Baumeister}
\address{Barbara Baumeister, Universit\"at Bielefeld, Germany}
\email{b.baumeister@math.uni-bielefeld.de}
\author[Gr\"uninger]{Matthias Gr\"uninger}
\address{Matthias Gr\"uninger, Universit\'e catholique de Louvain, Belgium}
\email{matthias.grueninger@uclouvain.be}

\begin{document}
  
\maketitle

\begin{abstract}
We clarify the notion of  effective equivalence and characterize geometrically the
effectively equivalent permutation groups. In particular, we present examples showing that the latter 
do  not correspond to affinely equivalent polytopes thereby answering Question~2.12 of ~\cite{Baumeister2007}.
We apply our characterization to our examples.
\end{abstract}
\section{Introduction}
The permutation polytopes are an interesting class of polytopes, see for instance \cite{Baumeister2007,
BS96, BHNP11, BHNP12, BaSt03, BLO11, DLY09, EFRS06}.
In~\cite{Baumeister2007} the authors also studied which groups lead to affinely equivalent polytopes.
Abstractly isomorphic permutation groups do not need to have affinely equivalent permutation polytopes:
 For instance 
$\langle (12), (34)\rangle$ and $\langle (12)(34),$ $ (13)(24)\rangle$ are isomorphic groups,
 but the associated permutation polytopes are a quadrangle and a 
tetrahedron, respectively, and therefore not affinely equivalent.

On the other hand, the notion of isomorphism of permutation groups is too restrictive to describe the
affine equivalent permutation polytopes.
In ~\cite{Baumeister2007}, Section 2.1,
it has been  observed that there are two permutation groups which are not isomorphic 
as permutation groups but whose permutation polytopes are affinely equivalent: 
The permutation polytopes of the  permutation groups 
$$\langle(1234) \rangle \leq \Sym(4), \langle(1234)(5)\rangle \leq \Sym(5),
\langle(1234)(56)\rangle  \leq \Sym(6)$$
are all tetrahedrons and therefore all affinely equivalent,
but the underlying groups are not isomorphic as permutation groups.

The notion of isomorphism of permutation groups has been generalized to the notion of 
effectively equivalent permutation groups  in \cite{Baumeister2007}; for the definition see the next section. In fact, all the 
permutation groups
listed in the last paragraph are effectively equivalent permutation groups. 

The hope was that two permutation groups are effectively
equivalent if and
only if the groups are isomorphic and the corresponding permutation polytopes are affinely equivalent.
Here we  present an example of two permutation groups which are isomorphic as abstract groups
and whose permutation polytopes are
affinely equivalent, but which are not effectively equivalent. 

Up to now the tools in studying permutation polytopes were mainly connected to convex geometry and group theory. In this note we provide a 
representation theoretical basis for the study of these polytopes.
Moreover, we use representation theory to prove  a criterion which determines when two permutation groups $G_1$ and $G_2$
are effective equivalent in 
terms of their permutation polytopes $P(G_1)$ and $P(G_2)$ (see Theorem~\ref{characterization}, Corollary~\ref{Criterion} and 
Theorem~\ref{conjugate}).

The organization of the note is as follows: We start by introducing  the notation and recall the relevant previous results.
In the third section we present the examples. The representation theoretical approach is presented in 
the fourth section. There we also characterize the effectively equivalent groups.  
These results are applied to  our examples in the last section.

\textbf{Acknowledgments:}  The authors like to thank for support by the DFG
through the SFB 701 ``Spectral Structures and Topological Methods in
Mathematics''. Moreover, they like to thank Benjamin Nill for his very useful comments which helped
to improve the paper.

\section{Notation and previous results}

The convex and the affine hull of a set $S$ in a real vector space
will be denoted by $\conv(S)$ and by $\aff(S)$, respectively.

\subsection{Permutation polytopes}
An injective homomorphism $\pi:G \to \Sym(n)$ is called {\em permutation representation}. 
The pair  $(G, \pi)$ is called {\em permutation group}.  In this case,  we obtain a  representation polytope as follows.

The symmetric group $\Sym(n)$ acts on the set $ \{1, \ldots , n\}$. Let $V$ be an $n$-dimensional $\R$-vector space
with basis $\{  e_1, \ldots, e_n \}$ and let $\Sym(n)$ act on this vector space by permuting the indices of the vectors in the  basis.
Then $V$ is the  {\em  permutation module for $\Sym(n)$}. This module induces a 
representation 
$R:\Sym(n) \to \GL(V) = \GL(\R^n)$, 
 $g \mapsto M_g$, and thereby  identifies the symmetric group $\Sym(n)$ 
with the set of $n \times n$ {\em permutation matrices}, i.e. the set of matrices whose entries are $0$ or $1$ 
such that in every column and every   row there is a unique $1$. 
The polytope 
\[P(G, \pi) : = \conv(\pi(G)) \subseteq \Mat_n(\R) \cong \R^{n \times n}\]
is called the {\em permutation polytope} associated to $(G, \pi)$.
If $\pi$ or $G$ are clear from the context then  we abbreviate $P(G, \pi)$ by $P(G)$ or by $P(\pi)$.

The special case $G=\Sym(n)$ yields the well-known $n$th {\em Birkhoff
  polytope} $B_n:=P(\pi(\Sym(n)))$ (see e.g. \cite{BS96}). 
  The concept of a permutation polytope can be generalized to a representation polytope by considering any 
  real representation of $G$ instead of a permutation representation, see ~\cite{Baumeister2007}.

\subsection{Notions of equivalence of polytopes.}

For a standard reference on polytopes we refer to \cite{Zie95}.
If the vertices of a polytope $R \subseteq \R^m$ are a subset of a full dimensional lattice $\Lambda$ 
in $\R^m$, then we call $P$ a lattice polytope. In this sense every permutation polytope is a lattice
polytope, as the vertices all lie in $\Mat_n(\Z)$.

 As moreover, every vertex of a permutation polytope is a matrix whose entries are only
$0$ and $1$, it is also a  {\em $0/1$-polytope}, i.e. a polytope whose vertices are in the set
$\{0,1\}^d$ for some $d \in \N$.

There are several notions of equivalence of (lattice) polytopes (see \cite{Zie95}):

\begin{defi}
{\rm 
Two polytopes $P \subset \R^m$ and $Q \subset \R^n$ are} affinely
equivalent {\rm if there is an affine isomorphism of the affine hulls $\phi
\colon \aff (P) \to \aff (Q)$ that maps $P$ onto $Q$, write $P \approx_{aff} Q$. For} lattice
equivalence {\rm we additionally require that $\phi$ is an isomorphism of
the affine lattices $(\aff P) \cap \Lambda \to ( \aff Q) \cap
\Lambda'$.} Combinatorial equivalence {\rm is an equivalence of the face
lattices as posets.}
\end{defi}

\subsection{Notions of equivalence of groups.}

To identify permutation groups that define affinely equivalent permutation polytopes  the notion of
effective equivalence has been introduced (see \cite{Baumeister2007}). 

For $K = \R$ or $\C$ we
denote by $\Irr_K(G)$ the set of pairwise non-isomorphic irreducible $K$-representations, i.e. homomorphisms
$G \to \GL(W)$ where $W$ is a $K$-vector space which does not contain a proper $G$-invariant subspace. 
For instance, there is the {\em trivial representation}, $1_G$: $G \to \GL(K)$, $g \mapsto 1$.  
Every representation   $\rho: G \to \GL(V)$ over $K$  splits into 
irreducible representations. We denote these {\em irreducible factors} of $\rho$ by $\Irr_K(\rho) \subseteq \Irr_K(G)$.

\begin{defi}\label{stable equivalence}
{\rm Two real representations $\rho_1$ and $\rho_2$ of $G$ are} stably equivalent {\rm if they contain the same non-trivial irreducible factors.
Two faithful real representations $\rho_i: G_i \rightarrow \GL(V_i)$ (for $i = 1,2$) of finite groups are}
effectively equivalent
{\rm if there exists an isomorphism
$\varphi:G_1 \rightarrow G_2$ such that  $\rho_1$ and $\rho_2 \circ \varphi$ are stably equivalent $G_1$-representations, write $\rho_1 \approx_{eff} \rho_2$.
Moreover, we say $G_1 \leq \Sym(n_1)$ and  $G_2 \leq \Sym(n_2)$ are} effectively equivalent permutation
groups {\rm if $G_1 \hookrightarrow \Sym(n_1)$ and $G_2 \hookrightarrow \Sym(n_2)$  are 
effectively equivalent permutation representations.}
\end{defi}

The permutation groups given in the introduction are effective equivalent. One may think that  two transitive permutation groups 
$G_1$ and $G_2$ which are effective equivalent  already have to be equal. This is not the case as demonstrated in the next example
(see also Example~\ref{almost}).
\bigskip\\
{\bf Example.}
Let  $G = \PSL_2(13)$ and consider the actions $\pi_1$ and $\pi_2$ on the coset spaces $G/H_1$ and $G/H_2$ where
$H_1$ is subgroup of $G$ isomorphic to $D_{14}$ and $H_2$ a subgroup isomorphic to $ D_{12}$. According to Atlas-notation the
permutation characters are $1a + 12abc + 13a + 14aa$
and $1a + 12abc + 13aa + 14aa$, respectively, see \cite{CNPW85}. Thus $\pi_1$ and $\pi_2$ are two transitive effective equivalent 
presentations which are different. 
\bigskip\\
An immediat{\bf e} consequence of the definition of effective equivalence is the following.

\begin{lemma}\label{equal}
 If $G_1=(G,\pi_1)$ and $G_2=(G,\pi_2)$ are two  permutation groups such that $P(\pi_1) = P(\pi_2)$, then $G_1$ and $G_2$
 are effective equivalent. 
\end{lemma}
\begin{proof}
 In this case $\pi_1^{-1}\pi_2$ is automorphism of $G$ and therefore, $G_1$ and $G_2$
are effective equivalent.
\end{proof}

In ~\cite[2.3]{Baumeister2007} we showed that if $\rho$ and $\bar{\rho}$ are two  stably equivalent real representations
of a finite group $G$, then $P(\rho)$ and $P(\bar{\rho})$ are affinely equivalent. If $\pi_1$ and $\pi_2$ are  effectively
equivalent  permutation representations, then  $\pi_1$ and $\pi_2 \circ \varphi$ are stably equivalent for some isomorphism
$\varphi: G_1 \rightarrow G_2$. As $P(\pi_2) = P(\pi_2 \circ \varphi)$ the following holds as well:

\begin{theorem}
The permutation polytopes related to two effectively
equi\-valent permutation representations are affinely equivalent.
\end{theorem}

Notice that Example~2.7 in \cite{Baumeister2007} shows that effectively
 equivalent permutation representations
do not necessarily have lattice equivalent permutation polytopes.
 This example shows as well that the volumes of two permutation polytopes associated
to effectively equivalent permutation representations may be different.

\section{The examples.}

In this section we present an example of a group with two non effectively equivalent
permutation representations such that the related permutation polytopes
are affinely equivalent. But first we show the following "almost example". It consists of two permutation groups  which are
not stably equivalent, but whose permutation polytopes are even equal.  It is not really an example to our question as the 
permutation groups are effectively equivalent.

\subsection{An "almost example".}\label{almost}

Let $G = \Alt(6)$. Then $G$ contains two different
subgroups $H_1$ and $H_2$ which are both isomorphic to $\Alt(5)$, but not conjugate in
$G$. We may choose $H_1$ as the stabilizer of $1$ in the action of $G$ on the set
$[6] := \{1, \ldots , 6\}$. Then $H_2$ is transitive on $[6]$. The group $G$ acts  on
both coset spaces $G/H_1$  and $G/H_2$, which yields two permutation representations $\pi_1$ and $\pi_2$. These representations are not stably equivalent, as they contain
different irreducible constituents, see for instance  \cite{CNPW85}, p. 4. On the other hand, as $|G:H_i| = 6$ for $i = 1,2$, both
representations $\pi_1$ and $\pi_2$ induce embeddings of $G$ into $\Sym(6)$. Since in $\Sym(6)$ there is only one subgroup 
isomorphic to $\Alt(6)$, it follows that $\pi_1(G) = \pi_2(G)$. Thus  by Lemma~\ref{equal}
the two representations $\pi_1$ and $\pi_2$ are effectively equivalent.
 \medskip
 \\
 Example~\ref{almost} provides two transitive effective equivalent permutation representations $\pi_1$ and $\pi_2$ of the same degree 
 which are different. But $\pi_1 = \pi_2 \circ \varphi$ for some $\varphi \in \Aut(G)$. We would like to know whether this is always the 
 case:
 
 \begin{quest}
 Do two  transitive effective equivalent permutation representations $(G,\pi_1)$ and $(G,\pi_2)$ of the same degree  always only differ
 by an automorphism of $G$?
 \end{quest}

\subsection{The example.}

Let $A = (\Z_2)^2$ be the direct product of two cyclic groups of order two, $B = \Z_4$ and $C = \Z_3$ cyclic groups of order $4$ and $3$,
and let $G := A \times B \times C$. 
In the following we define two different permutation representations of $G$:
\medskip\\
{\bf The permutation representation $\pi_1$.} Let $O_1$ be the disjoint union of the right coset spaces $O_{11}:=G/A$ and $O_{12}:=G/(B\times C)$
 and let  $G$ act by left multiplication on $O_1$.
Then $O_{11}$ and $O_{12}$ are the $G$-orbits. The kernels of the action of $G$ on $O_{11}$ and
$O_{12}$ are $A$ and $B\times C$, respectively. By Lemma~2.7 and Theorem~3.5 of \cite{Baumeister2007} $P(\pi_1)$ is the 
combinatorial product of an $11$-simplex with a $3$-simplex.

Notice, if $G = H \times K$ for some subgroups $H$ and $K$ of $G$, then we can extend every irreducible complex representation $\varphi_H$
of $H$  to an irreducible complex representation $\varphi$ of $G$ by sending every element of $K$ to the identity. Therefore, we can
embed $\Irr_\C(H)$ into $\Irr_\C(G)$.
In this sense  $\Irr_\C(\pi_1)$ is the union of $\Irr_\C(B\times C)$ and $\Irr_\C(A)$. As for an abelian group $\Irr_\C(G) \cong G$,
see Paragraph~6, 6.4 in \cite{Hup67}, it follows
that $\Irr_\C(\pi_1)$ is the union of two subgroups isomorphic to $\Z_4 \times \Z_3$ and $ (\Z_2)^2$, respectively.
\medskip\\
{\bf The permutation representation $\pi_2$.} Let $O_2$ be the disjoint union of the right coset spaces $O_{21}:=G/B$ and $O_{22}:=G/(A\times C)$ and 
let $G$ act  by left multiplication on $O_2$.
Then $O_{21}$ and $O_{22}$ are the $G$-orbits. The kernels of the action of $G$ on $O_{21}$ and
$O_{22}$ are $B$ and $A\times C$, respectively. By Lemma~2.7 and Theorem~3.5 of \cite{Baumeister2007} $P(\pi_2)$ is again the 
combinatorial product of an $11$-simplex with a $3$-simplex.

Here $\Irr_\C(\pi_2)$ is the union of $\Irr_\C(A\times C)$ and $\Irr_\C(B)$ and therefore, the union of two subgroups isomorphic to
$(\Z_2)^2 \times \Z_3$ and $\Z_4$.

It follows that $P(\pi_1)$ and $P(\pi_2)$ are affinely equivalent. In $\Irr_\C(\pi_1)$ there is an irreducible representation
of order $12$, while every element in $\Irr_\C(\pi_2)$ has order at most $6$. This shows that  the induced real representations
$\pi_1$ and $\pi_2\circ \varphi$ do not contain the same irreducible factors for every automorphism $\varphi \in \Aut(G)$. Thus $\pi_1$
and $\pi_2$ are not effectively equivalent.

\begin{remark}\label{Allg Beispiel}
If $G = A \times B\times C$ is an abelian group such that $A$ and $B$ are non-isomorphic groups of the same size,
 then we can always construct such an example with $\pi_1$ and $\pi_2$ the actions of $G$ on the unions of coset spaces 
 $O_1:= G/A \cup G/(B \times C)$ and $O_2:= G/B \cup G/(A \times C)$, respectively.
 \end{remark}

\section{Characterization of effectively equivalence.}

Let $G = (G, \pi)$ be a permutation group of degree $n$ with permutation module $V$.
Then the   affine hull  of the polytope $P= P(\pi)$ is
$$\aff(\pi(G)) = \{\sum_{g \in G} \lambda_g M_g \mid \lambda_g \in \R, \sum_{g \in G} \lambda_g = 1\} = 
E_n +U_\pi,~\mbox{where}$$
$$U_\pi:=  \{ \sum_{g\in G } \lambda_g  M_g  \mid\lambda_g \in \R, \sum_{g \in G} \lambda_g =0\} =
 \{\sum_{g \in G} \lambda_g (M_g - M_e) \mid \lambda_g \in \R\} $$
$$= \langle M_g-M_h~|~g , h \in G \rangle_\R \leq \End(V).$$
The $\R$-vector space $U_\pi$ is  a $G$-module through the definition
$$h M := M_h M = \pi(h)M,  \mbox{for}~h \in G~\mbox{and}~M \in U_\pi,$$
as $\pi$ is a group homomorphism from $G$ into $\Sym(n) \leq \GL_n(\R)$.
Notice, that if in particular $M = \pi(g)$, then $h M = \pi(hg)$.

In order to nicely describe the  structure of $U_\pi$ we introduce more notation. 
For $\chi \in \Irr_\R(G)$ let $V(\chi)$ be an irreducible $\R G$-module with character $\chi$. Then $\End_{\R G}(V(\chi))$ 
is either isomorphic to $\R$, $\C$ or to the quaternions $\HH$ and thus $d_{\chi}:= 1/\dim_{\R} \End_{\R G} (V(\chi))$ is either $1$,
 ${1 \over 2}$ or ${1 \over 4}$. 
Set 
$$V_\pi:=   \{ \sum_{g\in G } \lambda_g  M_g  \mid \lambda_g \in \R\}~\mbox{and}~\epsilon_\pi:= \sum_{g \in G} M_g.$$
Then $V_\pi = U_\pi \oplus \R \epsilon_G$.  

\begin{theorem}\label{Isotype}
Let $G = (G, \pi)$  be a permutation group of degree $n$ with permutation module $V$.
Then $U_\pi$ is isomorphic to $$\sum_{\chi \in \Irr_\R(\pi) \setminus \{1_G\}} (d_\chi \cdot \chi(1)) V(\chi)$$
as an $\R G$-module, where for  $d$ a natural number $dV(\chi)$ is the direct sum of $d$ to $V(\chi)$ isomorphic $\R G$-modules.
\end{theorem}
\begin{proof}
Extend $\pi$ linearly to an  $\R$-algebra epimorphism $\pi: \R G \to V_\pi$. This then is $\R G$-linear.  
By Maschke's Theorem, the group algebra $\R G$ is semi-simple. Thus $V_\pi$ is semi-simple as well, and 
by a theorem by Wedderburn  (see for instance \cite{Hup67} Chapter V, Hauptsatz~4.4) we get
$\R G = \oplus_{i=1}^m A_i$ with $A_i$ simple.  
Since $A_i$ is simple, 
either $\pi|A_i$ is injective or $\pi(A_i)=0$. Thus $V_\pi$ is semi-simple as well and we can assume that there is an integer $k \leq m$ such that 
$V_\pi \cong \oplus_{i=1}^k A_i$. Moreover, we can assume that $A_1 =\R \epsilon$ with 
$\epsilon =\sum_{g \in G} g$.
Since $\pi(\epsilon) =\epsilon_\pi$, we have $U_\pi =\oplus_{i=2}^k A_i$. 

Let $\chi \in \Irr_\R(G)$ be an irreducible  representation of $G$. Then $V(\chi)$ is a
composition factor of the $G$-module $V$
 if and only if there is an index 
 $1 \leq i \leq k$ such that the $\R$-linear extension of $\chi$ does not vanish on 
 $A_i$ or equivalently, $A_i = e_{\chi} \R G$ with
$$e_{\chi}=\sum_{\psi \in \Irr_{\C}(G|\chi)} e_{\psi} ~=   \sum_{\psi \in \Irr_{\C}(G|\chi)} {1 \over {|G|}}
\psi(1) \sum_{g \in G} \overline{\psi(g)} g$$ the central idempotent corresponding to $\chi$. In this case,
 $$\dim_{\R} A_i =\dim_{\R} End_{\R G}(V(\chi)) \cdot (\dim_{End_{\R G} (V(\chi))} V(\chi)))^2 $$ 
 $$= (1 /d_{\chi}) \cdot (d_{\chi} \cdot \chi(1))^2 = d_{\chi} \cdot \chi(1)^2$$
  and $A_i \cong \End_{\R}(V(\chi)) \cong(d_\chi \cdot \chi(1)) V(\chi)$ as 
 an $\R G$-module, see  \cite{Hup67}, Chapter V, Satz 4.5. Since $\epsilon$ is the central idempotent corresponding to the trivial 
 representation, the claim follows. 
 \end{proof}

The following lemma shows that affine maps between permutation polytopes are always induced by linear maps of their 
linear hulls.

\begin{lemma}\label{Fortsetzung} If $V$ is an $\R$-vectorspace and $\psi: \aff(P(\pi)) \to V$ an affine map, then $\psi$
can be uniquely lifted to a linear map $\Psi: V_{\pi} \to V$.
\end{lemma}
\begin{proof} Since $P(\pi)$ contains a basis of $V_\pi$, the space $V_\pi$ is the affine hull of $\{0\}$ and $P(\pi)$. Moreover, every 
element in $\aff(P(\pi))$ is a matrix whose rows and columns all have sum $1$. This shows that
$0 \not\in \aff(P(\pi))$. Thus there is a unique affine map $\Psi: V_\pi \to V$ such that $\Psi|\aff(P(\pi)) =\psi$ 
and $\Psi(0) =0$. Since $\Psi(0) =0$, the map $\Psi$ is linear.
\end{proof}

For $\varphi \in \Aut (G)$ we say that two $G$-modules  $U_1$ and $U_2$ are {\em $\varphi$-isomorphic} if there exists an isomorphism 
$\phi:U_{1} \to U_{2}$ with
$\phi (g u ) = \varphi(g) \phi(u)$ for all $u \in U_{1}$ and all $g\in G$.
Theorem~\ref{Isotype} implies the following characterization of the effectively equivalent permutation groups:

\begin{theorem}\label{characterization} Let $G_1 = (G, \pi_1)$ and $G_2 = (G,\pi_2)$ be two permutation groups. 
Then the following are equivalent:
\begin{enumerate}
\item[(a)] $G_1$ and $G_2$ are effectively equivalent.
\item[(b)] There is a $\varphi \in  \Aut (G)$ such that $U_{\pi_1}$ and $U_{\pi_2}$  are $\varphi$-isomorphic.
 \item[(c)]  There is a $\varphi\in  \Aut (G)$ such that $V_{\pi_1}$ and $V_{\pi_2}$  are $\varphi$-isomorphic.
\item[(d)]  There is an affine isomorphism $\phi: \aff(P(\pi_1)) \to \aff(P(\pi_2))$ which maps $P(\pi_1)$ onto 
$P(\pi_2)$ and which restricted to $\pi_1(G)$ is a group homomorphism.
\end{enumerate}
\end{theorem}
\begin{proof}
Suppose that (b) holds. Then $U_{\pi_1}$ and $U_{\pi_2 \circ \varphi}$ are isomorphic $G$-modules. 
By Theorem~\ref{Isotype} $U_{\pi_2 \circ \varphi}$ and 
$U_{\pi_1}$ have the same irreducible constituents. Thus $\pi_2\circ \varphi$ and $\pi_1$ are stably 
equivalent and $G_1$ and $G_2$ are effectively equivalent; so (a) holds.  Statements (b) and (c) are equivalent since 
$V_{\pi_i}$ and $U_{\pi_i}$ only differ by the trivial $\R G$-module. 

Suppose that (a) holds. 
Then there is an automorphism $\varphi$ of $G$ such that $\pi_1$ and $\pi_2 \circ \varphi$ are stably equivalent. 
Thus if $(\lambda_g)_{g \in G}$ is a family of real numbers with $\sum_{g \in G} \lambda_g =0$, 
then by \cite{Baumeister2007} Theorem~2.4
$$\sum_{g \in G} \lambda_g \pi_1 (g) =0~\mbox{ if and only if}~\sum_{g \in G} \lambda_g \pi_2
 (\varphi(g)) =0.$$
  Thus the map $$\phi: \aff(P(\pi_1)) \to \aff(P(\pi_2)), \sum_{g \in G} \lambda_g \pi_1(g) \mapsto 
 \sum_{g \in G} \lambda_g \pi_2(\varphi(g)),$$  
 where $  \sum_{g\in G} \lambda_g =1,$ is a well-defined affine map.

 By the same argument, $\phi$ is injective, and as the image of $\phi$ affinely 
 spans $\aff(P(\pi_2))$, the map $\phi$ is surjective as well. As $\pi_1,\pi_2$ and $\varphi$ are group homomorphisms, 
 the restriction of $\phi$ to $\pi_1(G)$ is a group homomorphism onto $\pi_2(G)$. This shows that (a) implies (d).
 
Suppose that (d) holds. We want to show (c). First note that $\phi$ maps $\pi_1(G)$ bijectively onto 
 $\pi_2(G)$, since these are vertices of the corresponding polytopes. Thus $\phi$
 induces a group isomorphism between $\pi_1(G)$ and $\pi_2(G)$ which we will also call $\phi$. 
 Then $\varphi:= \pi_2^{-1} \circ \phi \circ \pi_1$ is a group automorphism of $G$. If $u =\sum_{g \in G} \lambda_g
 \pi_1(g) $ with $\sum_{g\in G} \lambda_g =1$ and $x \in G$, then
 $$\phi(xu) =  \phi(x \sum_{g\in G} \lambda_g \pi_1(g) ) = \phi(\sum_{g\in G} \lambda_g x\pi_1(g) ) = \sum_{g\in G} \lambda_g \phi(\pi_1(xg))= $$
 $$\sum_{g\in G} \lambda_g \phi(\pi_1(x))\phi(\pi_1(g)) =  \sum_{g\in G} \lambda_g \pi_2(\varphi(x)) \phi(\pi_1(g))$$
 $$=  \sum_{g\in G} \lambda_g \varphi(x) \phi(\pi_1(g)) = \varphi(x)  \sum_{g\in G} \lambda_g \phi(\pi_1(g)) = \varphi(x)\phi( u).$$
 by the definition of the action of $G$ on $U_{\pi_1}$ and $U_{\pi_2}$.
By Lemma~\ref{Fortsetzung} we can extend $\phi$ to a linear isomorphism $\Psi: V_{\pi_1} \to V_{\pi_2}$, for which one easily
sees that $\Psi(gu) = \varphi(g) \Psi(u)$ holds for all $g\in G$ and all $u \in V_{\pi_1}$.  
\end{proof}

The equivalence between (a) and (d) yields another possibility to describe effective equivalence.

\begin{defi}
{\rm If $P$ and $Q$ 
are two polytopes, $G$ a group which acts as automorphism group on both $P$ and $Q$,
 then an affine isomorphism $\phi: \aff(P) \to \aff(Q)$ with $\phi(P) = Q$ is called an} affine $G$-isomorphism {\rm if 
 there is an automorphism $\varphi$ of $G$ such that 
 $\phi(gx) =\varphi(g)\phi(x)$ holds for all $g \in G$ and all vertices $x $ of $P$.}
 \end{defi}
 
If $\pi$ is a permutation representation of a finite group $G$, then left multiplication defines a natural 
action of $G$ on $P(\pi)$ as we saw above. Then the equivalence between (a) and (d) of ~\ref{characterization} gives us:

\begin{corollary}\label{Criterion} If $G_1=(G,\pi_1)$ and $G_2=(G,\pi_2)$ are two permutation groups, then they are effectively 
equivalent if and only if there is  an affine $G$-isomorphism between the 
corresponding permutation polytopes. 
\end{corollary} 

Let $G_i=(G,\pi_i)$ be permutation groups for $i=1,2$. Suppose that there is an 
affine isomorphism $\phi:P(\pi_1) \to P(\pi_2)$. Then we get an isomorphism $\hat{\phi}:
\Aut (P(\pi_1)) \to  \Aut (P(\pi_2))$ by setting $\hat{\phi}(g)(x) : =\phi(g(\phi^{-1}(x))$ for $g\in \Aut (P(\pi_1))$ 
and $x\in P(\pi_2)$. For $g\in G$ let $\psi_i(g) \in \Aut (P(\pi_i))$ be the automorphism 
defined by $\psi_i(g)(u)=\pi_i(g)u$ for $u \in \aff(P(\pi_i))$. Hence $\psi_i:G \to \Aut (P(\pi))$ 
is a monomorphism. Note that $\psi_i(G)$ 
acts regularly on the set of vertices of $P(\pi_i)$.

\begin{theorem}\label{conjugate}
 $G_1$ and $G_2$ are effectively equivalent if and only if there is an affine isomorphism 
$\phi:\aff(P(\pi_1)) \to \aff(P(\pi_2))$ mapping $P(\pi_1)$ onto $P(\pi_2)$ such that 
$\hat{\phi}(\psi_1(G))$ and $\psi_2(G)$ are conjugate in 
$\Aut (P(\pi_1))$.
\end{theorem}
\begin{proof}
If $G_1$ and $G_2$ are effectively equivalent, then by ~\ref{characterization} there is an 
affine isomorphism $\phi:\aff(P(\pi_1)) \to \aff(P(\pi_2))$ with 
$\Phi(P(\pi_1)) =P(\pi_2)$  and an isomophism $\varphi$ of $G$ with 
$\phi(xu) = \varphi(x) \phi(u)$ for all $x \in G$ and $u \in P(\pi_1)$. Therefore it follows 
immediately that $\psi_2(G) =\hat{\phi}(\psi_1(G))$. Now suppose there is  $\phi:\aff(P(\pi_1)) \to \aff(P(\pi_2))$ 
such that $\phi(P(\pi_1)) =P(\pi_2)$ and $\psi_2(G)$ and $\hat{\phi}(\psi_1(G))$ are conjugate. We may assume that 
$\phi(\pi_1(1)) =\pi_2(1)$. Since $\psi_2(G)$ acts transitively on the vertices of $P(\pi_2)$, we may if and only if
assume that there is an $a \in \Aut (P(\pi_2))_{\pi_2(1)} $ such that $\psi_2(G)^a = \hat{\phi}
(\psi_1(G))$.
After replacing $\phi$ by $ a^{-1} \circ \phi$ we may assume $\hat{\phi}(\psi_1(G)) =\psi_2(G)$.
Thus there is a bijective map $\varphi:G \to G$ such that $\psi_2(\varphi(g))=\hat{\phi}(\psi_2(g))$ 
for all $g\in G$. Since $\hat{\phi},\psi_1$ and $\psi_2$ are isomorphisms, it follows that 
$\varphi$ is an automorphism of $G$. Thus we have for all $g \in G$ and all $u \in \aff(P(\pi_1))$:
$$\phi(\psi_1(g)(u)) =\phi(\psi_1(g)(\phi^{-1}(\phi(u)))) =\hat{\phi}(\psi_1(g)(\phi(u))=
\psi_2(\varphi(g))(\phi(u)).$$
Since $\phi(\pi_1(g)) =\pi_2(g)$, we get 
$$\phi(\pi_1(g)) = \phi(\psi_1(g)(\pi_1(1)))=\psi_2(\varphi(g))\pi_2(1)=\pi_2(\varphi(g)).$$
Thus $\phi|\pi_1(G)$ is a group homomorphism and the claim follows by Theorem~\ref{characterization}.
\end{proof} 

As a consequence this lemma establishes  Conjecture~5.4 of \cite{Baumeister2007} for our examples.

\begin{corollary}
If $P(\pi_1) \approx_{aff} P(\pi_2)$, but $\pi_1 \not\approx_{eff} \pi_2$, then $\Aut (P(\pi_1)) \cong \Aut (P(\pi_2))$ contains two
non-conjugate regular subgroups which are isomorphic to $G$. In particular $\Aut (P(\pi_i))  >  \psi_i(G)$.
\end{corollary}
\begin{proof}
The polytopes $P(\pi_1)$ and $P(\pi_2)$ are affinely equivalent while $\pi_1$ and $\pi_2$ are not effectively equivalent. Thus by 
Theorem~\ref{conjugate} the subgroups 
$\hat{\phi}(\psi_1(G))$ and $\psi_2(G)$ of $\Aut (P(\pi_1))$ are not conjugate and therefore not equal, which yields the assertion.
\end{proof}

We wonder whether these results can be used to prove the following.

\begin{quest}
 Let $G_1=(G,\pi_1)$ and $G_2=(G,\pi_2)$ be two transitive permutation groups (of the same degree).
 Are then $P(\pi_1)$ and  $P(\pi_2)$ affinely equivalent if and only if 
 $\pi_1$ and $\pi_2$ are effectively equivalent?
 \end{quest}

This question has certainly a positive answer if $G$ is abelian, as in this case every transitive presentation is already regular.
 If $G_1$ and $G_2$ are not transitive then the answer to the question is negative as our examples demonstrate.
 Permutation polytopes for cyclic groups are of importance in statistics, see for instance \cite{EFRS06}. 
 Therefore, it would be interesting to know whether they behave 
 more nicely.
 
 \begin{quest}
 Let $G_1=(G,\pi_1)$ and $G_2=(G,\pi_2)$ be two cyclic permutation groups.
  Are then $P(\pi_1)$ and  $P(\pi_2)$ affinely equivalent if and only if 
 $\pi_1$ and $\pi_2$ are effectively equivalent?
 \end{quest}

\section{Applications}

\subsection{Application of Lemma~\ref{equal} to $\Alt(6)$}

We obtain as an immediat{\bf e} consequence the following well-known fact.

\begin{lemma}
 The group $G = \Alt(6)$ has an outer automorphism which interchanges the two conjugacy classes of groups of $G$
 which are isomorphic to $\Alt(5)$.
\end{lemma}
\begin{proof}
Let $H_1$ and $H_2$ be two subgroups of $\Alt(6)$ of the two conjugacy classes and $\pi_1$ and $\pi_2$ be 
as in Subsetion~\ref{almost}. Then, as both representations yield the same permutation polytope, 
Lemma~\ref{equal} implies that $(G,\pi_1)$ and $(G,\pi_2)$ are effective equivalent.
Thus $\pi_1^{-1}\pi_2$ is an automorphism of  $\Alt(6)$ that maps $H_1$ onto $H_2$. 
\end{proof}

\subsection{Application of Theorem~\ref{characterization} to the example}

In this section we apply our characterization of the effectively equivalent permutation groups given in 
Theorem~\ref{characterization} 
to give a new, direct and more geometric  proof  of the fact that the permutation groups $(G, \pi_1)$ 
and $(G, \pi_2)$ presented in  Example~3.2 are not effectively equivalent.
We continue to use the notation introduced in Example~3.2.

Suppose that $(G, \pi_1)$ and $(G, \pi_2)$ are effectively equivalent. Then according to 
Theorem~\ref{characterization}  there is an 
affine isomorphism $\phi: \aff(P(\pi_1)) \to \aff(P(\pi_2))$ which maps $P(\pi_1)$ to 
$P(\pi_2)$ and which restricted to $\pi_1(G)$ is a group homomorphism.

Let $H$ be a subgroup of $G$ such that $\pi_1(H)$ is the set of vertices of a face of the polytope $P(\pi_1)$.  Then $\pi_1(H)$ is a subgroup 
of $\pi_1(G)$ which implies that $\phi(\pi_1(H))$  is a subgroup of $\pi_2(G)$. As $\phi$ is an affine isomorphism 
from $\aff(P(\pi_1))$ to $ \aff(P(\pi_2))$ as well, the set $\phi(\pi_1(H))$ is the set of vertices of a face of the polytope $P(\pi_2)$.

Now we count all the faces of $P(\pi_i)$ which have $24$ vertices and whose set of vertices is a subgroup of $\pi_i(G)$ (for $i = 1,2$).
The polytope $P(\pi_i)$  is the product of an $11$-simplex with a $3$-simplex. Therefore every face of $P(\pi_i)$  has the shape
$E \times F$ where $E$ is a face of the $11$-simplex  and $F$ a face of the $3$-simplex (for $i = 1,2$).
\medskip\\
{\bf The faces of $P(\pi_1)$ given by a subgroup of size $24$.}  
In this case $G = H_1 \times H_2$ where $H_1 \cong \Z_{12}$ and $H_2 \cong \Z_2^2$; and $P(\pi_1) = P(H_1) \times P(H_2)$.
Further $P(H_1)$ is an $11$-simplex and $P(H_2)$ a $3$-simplex.
If $H$ is a subgroup of $G$ such that $\pi_1(H)$ is the set of vertices of a face with $24$ vertices, then $H = K_1 \times K_2$ 
such that  $K_i$ is a subgroup of $H_i$ (for $i = 1,2$) and such that $|H| = |K_1|\cdot |K_2| = 24$.  Then either $K_1 = H_1$
 and $K_2$ of order $2$ or $K_1$ is of order $6$ and $K_2 = H_2$. As there is just one
 subgroup of order $6$ in $H_1$ and three subgroups of order $2$ in $H_2$, it follows that  the $24$-vertex faces which are coming from a subgroup 
 are precisely three faces of the shape of a prisma over an $11$-simplex  and  one face which is the product of a $5$-simplex with a $3$-simplex.
\medskip\\
{\bf The faces of $P(\pi_2)$  given by a subgroup of size $24$.}  
 Here we have the factorization $G = M_1 \times M_2$ where $M_1 \cong \Z_2^2 \times 3$ and $M_2 \cong \Z_4$. The polytope
 $P(\pi_2) = P(M_1) \times P(M_2)$ is the product of an $11$-simplex and a $3$-simplex.  If $H$ is a subgroup of $G$ 
 such that $\pi_2(H)$ is the set of vertices of a face with $24$ vertices, then as above $H = K_1 \times K_2$ 
such that  $K_i$ is a subgroup of $M_i$ (for $i = 1,2$) and such that $|H| = |K_1|\cdot |K_2| = 24$.  In this case there are three
subgroups of $M_1$ of size $6$ and precisely one subgroup of $K_2$ of size $2$. Therefore,  the $24$-vertex faces which are coming from a subgroup  are precisely one prisma over an $11$-simplex  and three faces which are the product of a $5$-simplex and a $3$-simplex.
\medskip\\
This contradicts the fact that $\phi$ maps every  face  of $P(\pi_1)$ which is induced by a subgroup of $G$ isomorphically onto 
a face  of $P(\pi_2)$ which is induced by a subgroup of $G$.
Thus $(G, \pi_1)$ and $(G, \pi_2)$ are not effectively equivalent.

\bibliographystyle{alpha}

\end{document}